\documentclass[reqno,12pt]{amsart}

\usepackage{graphicx}
\usepackage{amssymb}
\usepackage{amsmath}
\usepackage{mathtools}
\usepackage{enumitem}
\usepackage[matrix,arrow,curve]{xy}

\theoremstyle{plain}
\newtheorem{theorem}{Theorem}

\newtheorem{corollary}{Corollary}

\newtheorem{lemma}{Lemma}

\theoremstyle{remark}

\newtheorem*{question}{Question}

\def\N{\mathbb N}
\def\Z{\mathbb Z}

\def\Q{\mathbb Q}
\def\x{\times}
\def\phi{\varphi}

\def\leq{\leqslant}
\def\geq{\geqslant}
\def\emptyset{\varnothing}
\def\rank{\operatorname{rk}}
\def\im{\operatorname{im}}
\def\ker{\operatorname{ker}}
\def\HF{\operatorname{HF}}
\def\HFhat{\widehat{\operatorname{HF}}}
\def\HFKhat{\widehat{\operatorname{HFK}}}

\title{Tight fibred knots without L-space surgeries}

\author{Filip Misev}
\address{Max Planck Institute for Mathematics, Bonn, Germany}
\email{fmisev@mpim-bonn.mpg.de}

\author{Gilberto Spano}
\address{LMNO, Universit\'e de Caen-Normandie, Caen, France}
\email{gilbertospano.math@gmail.com}

\date{}

\begin{document}
\maketitle

\begin{abstract}
We show there exist infinitely many knots of every fixed genus $g\geq 2$ which do not admit surgery to an L-space, despite resembling algebraic knots and L-space knots in general: they are algebraically concordant to the torus knot $T(2,2g+1)$ of the same genus and they are fibred and strongly quasipositive.
\end{abstract}

\thispagestyle{empty}

\section{Introduction and statement of result}
Algebraic knots, which include torus knots, are L-space knots: they admit Dehn surgeries to L-spaces, certain $3$-manifolds generalising lens spaces which are defined in terms of Heegaard-Floer homology \cite{He2}.

The first author recently described a method to construct infinite families of knots of any fixed genus $g\geq 2$ which all have the same Seifert form as the torus knot $T(2,2g+1)$ of the same genus, and which are all fibred, hyperbolic and strongly quasipositive. Besides all the classical knot invariants given by the Seifert form, such as the Alexander polynomial, Alexander module, knot signature, Levine-Tristram signatures, the homological monodromy (in summary, the algebraic concordance class), further invariants such as the $\tau$ and $s$ concordance invariants from Heegaard-Floer and Khovanov homology fail to distinguish these knots from the $T(2,2g+1)$ torus knot (and from each other). This is described in the article~\cite{Mis}, where a specific family of pairwise distinct knots $K_n$, $n\in\N$, with these properties is constructed for every fixed genus $g\geq 2$.

Here we show that none of the $K_n$ is an L-space knot (except $K_0$, which is the torus knot $T(2,2g+1)$ by construction). This implies our main result:

\begin{theorem} \label{thm:L-space}
For every integer $g\geq 2$, there exists an infinite family of pairwise distinct genus $g$ knots $K_n$, $n\in\N$, with the following properties.
\begin{enumerate}
\item $K_n$ is algebraically concordant to the torus knot $T(2,2g+1)$
\item $K_n$ is fibred, hyperbolic and strongly quasipositive
\item $K_n$ does not admit any nontrivial Dehn surgery to a Heegaard-Floer L-space
\end{enumerate}
\end{theorem}

We show in fact that the knots $K_n$ constructed in~\cite{Mis} do not have the same knot Floer homology as $T(2,2g+1)$.

Below we briefly introduce the notions of L-spaces, L-space knots, quasipositivity and fibredness and relate our result to recent work by Boileau, Boyer and Gordon on the subject of L-space knots. Section~\ref{sec:fixed points and exact triangles} contains a description of the fibred knots $K_n$ via their monodromy (taken from~\cite{Mis}) and collects the main ingredients from Heegaard-Floer theory and Lagrangian Floer homology used in the proof of our result, which is given in Section~\ref{sec:proof}.

\subsection{L-spaces, quasipositivity and fibredness}
\mbox{L-spaces} are named after {\em lens spaces}, three-dimensional manifolds formed by glueing two solid tori along their boundaries. By definition, a closed $3$-manifold $M$ is an {\em L-space} if it is a rational homology sphere, that is, $H_*(M,\Q)\cong H_*(S^3,\Q)$, and its Heegaard Floer homology has the smallest possible rank: $\rank\HFhat(M)=|H_1(M,\Z)|$. Every lens space (except $S^1\x S^2$, which fails to be a rational homology sphere) is in fact an L-space. More generally, $3$-manifolds with finite fundamental group (the manifolds with elliptic geometry, certain Seifert fibred manifolds) are known to be L-spaces~\cite[Proposition~2.3]{OzSz1}.


A knot $K\subset S^3$ is an {\em L-space knot} if some non-trivial integral Dehn surgery on $K$ yields an L-space. Basic examples include the torus knots and the Berge knots, since they admit lens space surgeries; see~\cite{Mos, Ber, Gre}. In addition, all algebraic knots (the connected links of plane curve singularities, which include all positive torus knots) are L-space knots. This follows from a theorem of Hedden stating that certain cables of L-space knots are again L-space knots~\cite[Theorem~1.10]{He2}, combined with the classical description of algebraic knots as iterated cables of torus knots, involving the Puiseux inequalities (see for example~\cite{EN}).

By work of Ghiggini~\cite{Ghi} and Ni \cite[Corollary~1.3]{Ni1}, \cite{Ni2}, all L-space knots are known to be fibred: they arise as the bindings of open book decompositions of $S^3$. In addition, Hedden proved that the open book associated to an L-space knot (or to its mirror) supports the tight contact structure of $S^3$, under Giroux' correspondence (see~\cite[Theorem~1.2, Proposition~2.1]{He1} and Ozsv\'ath-Szab\'o \cite[Corollary~1.6]{OzSz1}). In~summary, L-space knots are {\em tight fibred} knots.

Besides Dehn surgery, another important construction of $3$-manifolds, starting from a knot $K\subset S^3$, is given by branched covering of $S^3$ branched along $K$. The $n$-fold cyclic branched covering, branched along $K$, is denoted $\Sigma_n(K)$. For example, if $K$ is an alternating knot, $\Sigma_2(K)$ is always an L-space~\cite{OzSz2}.

Recently, Boileau, Boyer and Gordon \cite{BBG} considered L-space knots $K\subset S^3$ with the additional property that $\Sigma_n(K)$ is also an L-space for some $n\in\N$, $n\geq 2$. They deduced strong restrictions on such knots. In order to state their results, we briefly recall the notion of strongly quasipositive knots, which were introduced and first studied by Rudolph in the 80s~\cite{Ru}. Let the symbols $\sigma_1,\sigma_2,\ldots,\sigma_{n-1}$ denote the standard positive generators of the braid group on $n$ strands, $\sigma_i$ corresponding to the braid in which the $i$-th strand crosses over the $(i+1)$-st strand in the direction of the braid's orientation (and no other crossings). A braid $\beta$ is called {\em strongly quasipositive} if it can be written as a product of {\em conjugates} of the $\sigma_i$:
\[ \beta=\prod_{j=1}^d w_j\sigma_{n_j}w_j^{-1}, \]
where $d\in\N$, $n_1,\ldots,n_d\in\{1,\ldots,n-1\}$ and the conjugating words $w_j$ are of the special form
\[ w_j=\sigma_{j-k}\cdots\sigma_{j-1} \]
for some $k$ (depending on $j$). Accordingly, such braids are called strongly quasipositive. Hedden showed that a fibred knot is tight if and only if it is strongly quasipositive~\cite{He1}. Fibredness is important here: there do exist non-fibred strongly quasipositive knots.

\begin{theorem}[Boileau, Boyer, Gordon~\cite{BBG}]
Let $K$ be a strongly quasipositive knot with monic Alexander polynomial. Then
\begin{enumerate}
\item $\Sigma_n(K)$ is not an L-space for $n\geq 6$.
\item If $\Sigma_n(K)$ is an L-space for $2\leq n\leq 5$, then $K$ has maximal signature and its Alexander polynomial is a product of cyclotomic polynomials.
\end{enumerate}
\end{theorem}

\begin{corollary}[Boileau, Boyer, Gordon~\cite{BBG}] \label{cor:bbg}
Let $K$ be an L-space knot such that $\Sigma_n(K)$ is an L-space for some $n$. Then
\begin{enumerate}
\item $n\geq 4$ implies that $K$ is the trefoil knot.
\item $n=3$ implies that $K$ is either the trefoil knot or its Alexander polynomial is equal to $t^4-t^3+t^2-t+1$, the Alexander polynomial of the cinquefoil knot. If it is neither the trefoil nor the cinquefoil, it is a hyperbolic knot.
\end{enumerate}
\end{corollary}

These results suggest that the two properties, admitting an L-space surgery and admitting an L-space branched cover are orthogonal in the sense that only few knots seem to satisfy both properties simultaneously. When applied to genus $g$ knots $K_n$ satisfying the properties of Theorem~\ref{thm:L-space} (and described in the next section), the above Corollary~\ref{cor:bbg} implies that $\Sigma_m(K_n)$ is not an L-space for $m,g\geq 3$ and $n\neq 0$:

\begin{corollary}
Fix $g\geq 3$ and let $K_n$ denote the genus $g$ knot described below. If $m\geq 3$ and $n\neq 0$, the $m$-fold cyclic branched cover $\Sigma_m(K_n)$ of $S^3$, branched along $K_n$ is not an L-space.\hfill $\Box$
\end{corollary}

For $g=2$, the knots $K_n$ have the same Alexander polynomial as the cinquefoil knot $T(2,5)$. The following (open) question was brought to our attention by Ken Baker, Michel Boileau, Marco Golla and Arunima Ray, independently. 

\begin{question}
Is the double branched cover $\Sigma_2(K_n)$ an L-space for any of the knots $K_n$, $n\neq 0$ (mentioned in Theorem~\ref{thm:L-space} and described below)? Is $\Sigma_3(K_n)$ an L-space for any of the genus $g=2$ knots $K_n$?
\end{question}

\subsection*{Acknowledgements}
The authors would like to thank Michel Boileau and Paolo Ghiggini for many useful discussions. FM thanks the Max Planck Institute for Mathematics, Bonn, for its support and hospitality, where part of this work was completed. GS thanks the Laboratoire de Math\'{e}matiques Nicolas Oresme, Caen, for its support and hospitality. 

\section{Monodromies, L-spaces and exact triangles} \label{sec:fixed points and exact triangles}

\subsection{The monodromy of the fibred knots $K_n$} \label{sec:monodromy}
Let us recall the construction of the knots $K_n$ from the article~\cite{Mis}. Throughout, we fix an integer $g\geq 2$, the genus of the knots to be constructed. The fibred knots $K_n$ are given in terms of their monodromies: $K_n$ has a genus $g$ fibre surface~$S_n$. Since the topological type of $S_n\subset S^3$ does not depend on $n$, we can identify $S_n$ with a fixed abstract (non-embedded) surface $S$ and consider the monodromy of $K_n$ as a mapping class $\phi_n:S\to S$. Given a simple closed curve $\gamma\subset S$, we denote $t_\gamma:S\to S$ the right Dehn twist on~$\gamma$. Using this notation, the monodromy $\phi_n$, $n\in\N$, is defined as the following composition of Dehn twists:
\[ \phi_n := (t_{\beta_{g,n}}\circ t_{\beta_{g-1}}\circ\ldots\circ t_{\beta_1})\circ (t_{\alpha_g}\circ\ldots\circ t_{\alpha_1}), \]
where
\[ \beta_{g,n} := t_c^n(\beta_g) \]
and $\alpha_i, \beta_j$ and $c$ are the simple closed curves shown in Figure~\ref{fig:S}.

\begin{figure}[ht]
\includegraphics[scale=1.4]{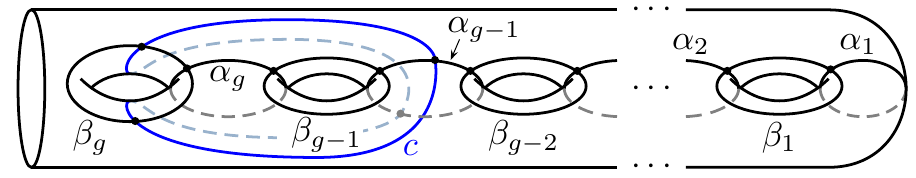}
\caption{The surface $S$ of genus $g$ and one boundary circle with the twist curves $\alpha_i$, $\beta_j$.}
\label{fig:S}
\end{figure}

\noindent The curve $c$, shown in blue, is the boundary of a neighbourhood of $\alpha_g\cup\beta_{g-1}$ in $S$. In particular, it is nullhomologous, intersects $\beta_g$ and $\alpha_{g-1}$ in exactly two points each, and does not intersect any of the remaining curves. It follows that $\beta_{g,n}\cap\alpha_{g-1}$ consists of $4n$ points (see Figure~\ref{fig:intersection}). Since $\beta_g$ and $\alpha_g$ intersect in exactly one point and $\alpha_g$, $c$ are disjoint, $\beta_{g,n}\cap\alpha_g$ is a singleton. Moreover, all pairs of curves involved in the construction realise their minimal geometric intersection number in their homotopy classes. This is clear whenever two curves are disjoint or intersect transversely in a single point. 
For the remaining cases, use the {\em bigon criterion} \cite[Proposition~1.7]{FM}. In particular, this applies to the curves $\beta_{g,n}$ and $\alpha_{g-1}$, whose minimal geometric intersection number $\iota(\alpha_{g-1},\beta_{g,n})=\#(\alpha_{g-1}\cap\beta_{g,n})=4n$ is used in the proof of our Theorem~\ref{thm:L-space}. These two curves are represented in Figure~\ref{fig:intersection} (the shaded band represents $2n$ parallel strands with alternating orientations which are part of $\beta_{g,n}$).

\begin{figure}[ht]
\includegraphics[scale=1.3]{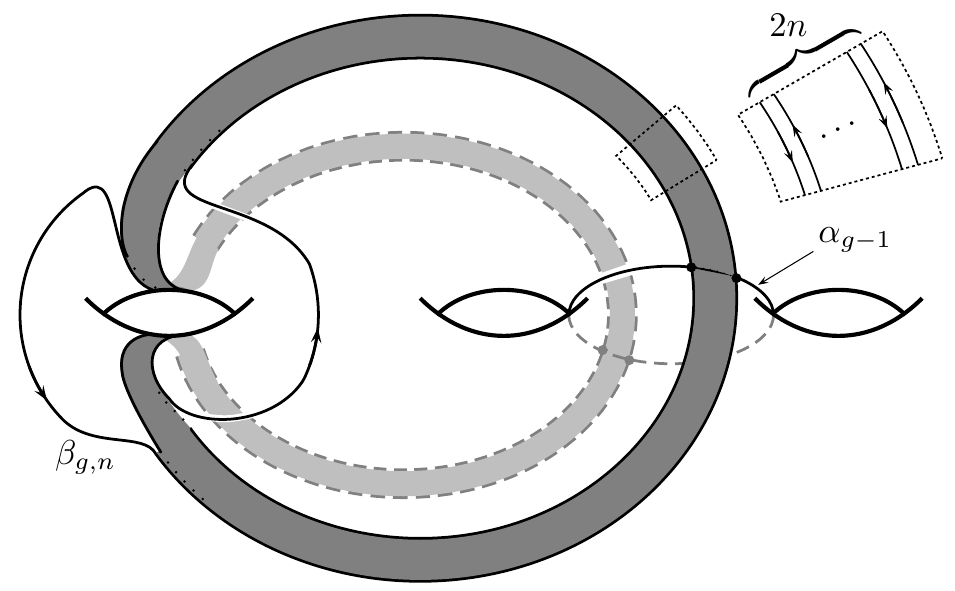}
\caption{The curve $\beta_{g,n}=t_c^n(\beta_g)$ wraps $n$-times around the curve $c$ in both directions which gives a total of $2n$ strands parallel to $c$ (represented by the shaded band). $\alpha_{g-1}$ and $\beta_{g,n}$ intersect in $4n$ points.}
\label{fig:intersection}
\end{figure}
 
\subsection{Floer homology of L-space knots}
The proof of our result, Theorem~\ref{thm:L-space}, relies on the following theorem by Ozsv\'ath-Szab\'o, which implies that the knot Floer homology groups of an L-space knot are at most one-dimensional in each Alexander degree.

\begin{theorem}[Ozsv\'ath-Szab\'o, Theorem~1.2 in~\cite{OzSz1}] \label{thm:OzSz}
Let $K\subset S^3$ be a knot which admits a positive integral L-space surgery. Then, there exists a sequence of integers $n_0=0<n_1<\ldots<n_k$ such that
\[ \HFKhat_d(S^3,K;j)\cong\left\{ \begin{array}{ll} \Z & \text{if $(j,d)=(\pm n_i,\delta_i)$, for some $i$} \\ 0 & \text{otherwise,} \end{array} \right. \]
where the supporting dimensions $\delta_i$ only depend on the $n_i$, according to the recursive formula 
\[ \delta_i = \left\{ \begin{array}{ll} 0 & \text{if}\ i = k \\
\delta_{i+1} - 2(n_{i+1}-n_i)+1 & \text{if $k-i$ is odd} \\
\delta_{i+1} - 1 & \text{if $k-i>0$ is even.} \end{array} \right. \]
\end{theorem}

We make use of two exact triangles in knot Floer homology and in Lagrangian Floer homology to bound the rank of the knot Floer homology groups of our knots $K_n$ from below. Since the $K_n$ are all fibred of genus $g$, their knot Floer homology groups in Alexander degree $\pm g$ have rank one. To prove that $K_n$ is not L-space for $n>0$ we will show that
\[\rank\HFKhat(S^3,K_n;-g+1) > 1,\]
which contradicts the condition of last theorem. Note also that it does not suffice to consider the coefficients of the Alexander polynomial of the genus $g$ knot $K_n$, since it equals the Alexander polynomial of the torus knot $T(2,2g+1)$, whose coefficients are all $\pm 1$.

\subsection{Two exact triangles in knot Floer homology and in Lagrangian Floer homology}
Let $Y$ be a closed oriented $3$--manifold and let $K \subset Y$ be a genus $g$ fibred knot with associated monodromy $\varphi \colon S \rightarrow S$, so that, if $\mathcal{N}$ is a tubular neighbourhood of $K$, $Y \setminus \mathcal{N}$ is homeomorphic to the mapping torus $(S \times [0,1])/((x,1) \sim (\varphi(x),0))$. An essential simple closed curve $\gamma \subset S$ can then be identified with a knot $L_{\gamma} \subset S \times \{0\} \subset Y$. 
 
Fix $\gamma$ and assume that $L_{\gamma}$ is non-trivial. Let $K' \subset Y' \coloneqq Y_{-1}(L_{\gamma})$ and, respectively, $K_0\subset Y_0 \coloneqq Y_0(L_{\gamma})$ be the knots determined by $K$ in the manifolds obtained via $-1$ and, respectively, $0$-surgery on $L_{\gamma}$. $K'$~is also a genus $g$ fibred knot, with associated fiber $S$ and monodromy $\varphi \circ t_{\gamma}$. On the other hand $K_0$ is not fibred and has genus $g-1$. If $F \subset Y_0$ is a genus $g-1$ Seifert surface of $K_0$, then it is easy to see that $(S,\gamma,\varphi(\gamma))$ is a \textit{sutured Heegaard diagram} for the sutured manifold $Y_0(F)$ (see the definitions 2.4 and 2.10 of \cite{Ju1}). Theorem 1.5 of \cite{Ju1} and the definition of sutured Floer homology imply that
\begin{equation} \label{Isomorphism HFK of K_0 and HF}
 \HFKhat(Y_0,K_0;-g+1) \cong \HF(\gamma,\varphi(\gamma)),
\end{equation}
 where the latter is the \emph{Lagrangian Floer homology} of $(\gamma,\varphi(\gamma))$. This is a homology whose generators of the chain complex are intersection points of $\gamma$ and $\varphi(\gamma)$ (which are assumed to be transverse; see for example \cite{Sei} for the details of the definition). We recall also that if two curves $a$ and $b$ are homotopic then
 \begin{equation}\label{Equation: rk = 2}
  \rank\HF(a,b)=2
 \end{equation}
 and if they are not homotopic then 
 \begin{equation}\label{Equation: rk = iota}
  \rank\HF(a,b) = \iota(a,b),
 \end{equation}
where $\iota(a,b)$ denotes the geometric intersection number of the curves $a$ and $b$.

\begin{lemma} \label{lem:ExactTriangleHFK}
There is an exact triangle
\[
\xymatrix{
\HFKhat(Y,K;-g+1) \ar[rr] && \HFKhat(Y',K';-g+1) \ar[dl] \\
& \HF(\gamma,\varphi(\gamma)). \ar[ul] }
\]
\end{lemma}

\begin{proof}
This is a direct consequence of \eqref{Isomorphism HFK of K_0 and HF} and Ozsv\'ath and Szab\'o's exact sequence for knot Floer homology.
\end{proof}

The second lemma we are going to use is a special case of a theorem due to Seidel. It generalises the well-known identity $\iota(t_a(b),b)=\iota(a,b)^2$ for essential simple closed curves $a,b$ on a surface \cite[Proposition~3.2]{FM}.

\begin{lemma}[Seidel~\cite{Sei}] \label{lem:Seidel}
Let $S$ be a compact oriented surface with boundary and $t_c:S\to S$ the Dehn twist on a simple closed curve $c\subset S$. For any pair of simple closed curves $a,b\subset S$, there is an exact triangle of Lagrangian Floer cohomology groups 
\[ 
\xymatrix{
\HF(t_c(a),b) \ar[rr] && \HF(a,b) \ar[dl] \\
& \HF(c,a) \otimes \HF(b,c) \ar[ul] }
\]
\end{lemma}

\section{Proof of the main theorem} \label{sec:proof}
\begin{proof}[Proof of Theorem~\ref{thm:L-space}]
Let $\psi\coloneqq t_{\beta_{g-1}}\circ\ldots\circ t_{\beta_1}\circ t_{\alpha_g}\circ\ldots\circ t_{\alpha_1}$, so that $\phi_n=t_{\beta_{g,n}}\circ\psi$. Let $(Y,K)$ be given by the open book $(S,\psi)$. That is, $K\subset Y$ is a fibred knot with monodromy $\psi:S\to S$. Note that both the $3$-manifold $Y$ and the knot $K\subset Y$ are independent of $n$, since $S$ only depends on $g$ and $\psi$ is also independent of $n$. Now consider the open book $(S^3,K_n)$ associated to the fibred knot $K_n$, which is the knot of interest. It is obtained from $(Y,K)$ by $(-1)$-Dehn surgery along the curve $\beta_{g,n}\subset S\subset Y$. We can therefore apply Lemma~\ref{lem:ExactTriangleHFK} to this situation, where $(Y',K')=(S^3,K_n)$.
\[ \xymatrix{\HFKhat(Y,K;-g+1) \ar[rr]^x && \HFKhat(S^3,K_n;-g+1) \ar[dl]^y \\ & \HF(\beta_{g,n},\psi(\beta_{g,n})) \ar[ul]^z} \]
Exactness and the rank-nullity formula for $x,y$ and $z$ imply
\begin{eqnarray}
\rank\HFKhat(S^3,K_n;-g+1) &=& \dim\ker y + \dim\im y \nonumber \\
&=& \dim\im x + \dim\ker z \nonumber \\
&=& \rank\HFKhat(Y,K;-g+1) - \dim\ker x \nonumber \\
&&+ \rank\HF(\beta_{g,n},\psi(\beta_{g,n})) - \dim\im z \nonumber \\
&\geq & \rank\HF(\beta_{g,n},\psi(\beta_{g,n})) - \rank\HFKhat(Y,K;-g+1) \nonumber
\end{eqnarray}
To calculate the rank of $\HFKhat(Y,K;-g+1)$, we apply Lemma~\ref{lem:ExactTriangleHFK} again, now taking $(Y',K')$ to be the open book obtained from $(Y,K)$ by $(-1)$-Dehn surgery along $\beta_{g,0}=\beta_g\subset S\subset Y$. We obtain $Y'=S^3$ and $K'=K_0$. The exact triangle from Lemma~\ref{lem:ExactTriangleHFK} now reads as follows.
\[ \xymatrix{\HFKhat(Y,K;-g+1) \ar[rr] && \HFKhat(S^3,K_0;-g+1) \ar[dl] \\ & \HF(\beta_g,\psi(\beta_g)) \ar[ul]} \]
Since $K_0=T(2,2g+1)$, we know that $\rank\HFKhat(S^3,K_0;-g+1)=1$. Further, the curve $\psi(\beta_g)=t_{\beta_{g-1}}(t_{\alpha_g}(\beta_g))$ intersects $\beta_g$ in exactly one point, whence $\rank\HF(\psi(\beta_g),\beta_g)=\iota(\beta_g,\psi(\beta_g))=1$. Because the above triangle is exact, we deduce 
\begin{equation} \label{Equation: rank of K less than 2}
 \rank\HFKhat(Y,K;-g+1) \leq 2. 
\end{equation}

The missing piece of information is $\rank\HF(\beta_{g,n},\psi(\beta_{g,n}))$. To compute it, observe first that
\[ \psi(\beta_{g,n})=t_{\beta_{g-1}}\circ t_{\alpha_g}\circ t_{\alpha_{g-1}}(\beta_{g,n}), \] because $\beta_{g,n}=t_c^n(\beta_g)$ does not intersect any of the curves $\alpha_i,\beta_i$ for $i\leq g-2$. We could, in principle, directly compute the intersection number by studying the curves $\beta_{g,n}$ and $\psi(\beta_{g,n})$ on the surface $S$. But Lemma~\ref{lem:Seidel} conveniently helps to simplify the calculation. First let us apply it to the curves $(a,b,c)=(\beta_{g,n},\alpha_g,\alpha_{g-1})$, which gives 
\begin{equation}\label{Equation: rank 1 part}
 \rank\HF(t_{\alpha_{g-1}}(\beta_{g,n}),\alpha_g)=\rank\HF(\beta_{g,n},\alpha_g) = 1
\end{equation}
because $\alpha_{g-1} \cap \alpha_g = \emptyset$ (so that the lower term of the exact triangle has rank $0$) and $\beta_{g,n}$ and $\alpha_g$ intersect in a single point.

Now apply Lemma~\ref{lem:Seidel} three times to estimate $\rank \HF(\beta_{g,n},\psi(\beta_{g,n})) = \rank\HF(\psi(\beta_{g,n}),\beta_{g,n})$:

\begin{enumerate}[leftmargin=*]
\item[1.] First apply the lemma to $(a,b,c)=(t_{\alpha_g}t_{\alpha_{g-1}}(\beta_{g,n}),\beta_{g,n},\beta_{g-1})$. Since $\beta_{g-1}\cap\beta_{g,n}=\emptyset$, the lower corner of the exact triangle vanishes and we get:
\[\rank\HF(\psi(\beta_{g,n}),\beta_{g,n})=\rank \HF(t_{\alpha_g} \circ t_{\alpha_{g-1}}(\beta_{g,n}),\beta_{g,n}).\]

\item[2.] Next choose $(a,b,c)=(t_{\alpha_{g-1}}(\beta_{g,n}),\beta_{g,n},\alpha_g)$. The lower term in the exact triangle is $\HF(\alpha_g,t_{\alpha_{g-1}}(\beta_{g,n})) \otimes \HF(\beta_{g,n},\alpha_g)$ which has rank~$1$ by \eqref{Equation: rank 1 part}. The exactness of the triangle implies that:
\[\rank \HF(t_{\alpha_g} \circ t_{\alpha_{g-1}}(\beta_{g,n}),\beta_{g,n}) \geq \rank \HF(t_{\alpha_{g-1}}(\beta_{g,n}),\beta_{g,n}) - 1.\]

\item[3.] Finally, choose $(a,b,c)=(\beta_{g,n},\beta_{g,n},\alpha_{g-1})$. The upper right term in the corresponding triangle is $\HF(\beta_{g,n},\beta_{g,n})$ which has rank~$2$ by equation \eqref{Equation: rk = 2}. It follows that:
\begin{eqnarray*}
 \rank \HF(t_{\alpha_{g-1}}(\beta_{g,n}),\beta_{g,n}) & \geq& \rank \left(\HF(\alpha_{g-1},\beta_{g,n}) \otimes \HF(\beta_{g,n},\alpha_{g-1})\right) - 2\\
						      & =& (\rank \HF(\alpha_{g-1},\beta_{g,n}))^2 -2. 
\end{eqnarray*}
\end{enumerate}
Summing up we get
\[\rank\HF(\psi(\beta_{g,n}),\beta_{g,n}) \geq (\rank\HF(\alpha_{g-1},\beta_{g,n}))^2 -3.\]

On the other hand the computation in Section~\ref{sec:monodromy} gives $\rank\HF(\alpha_{g-1},\beta_{g,n}) = \iota(\alpha_{g-1},\beta_{g,n}) =4n$, so that:
\[\rank\HF(\psi(\beta_{g,n}),\beta_{g,n}) \geq 16n^2 -3.\]

Finally, substituting the last estimation and the one in \eqref{Equation: rank of K less than 2} in the rank inequality obtained at the beginning of this section we get:
\begin{eqnarray*}
 \rank\HFKhat(S^3,K_n;-g+1) &\geq& \rank\HF(\beta_{g,n},\psi(\beta_{g,n})) - \rank\HFKhat(Y,K;-g+1) \\
			    &\geq& 16n^2-5
\end{eqnarray*}
If $n \neq 0$ this quantity is strictly grater than $1$ and Theorem~\ref{thm:OzSz} implies that $K_n$ cannot be an L-space knot. This establishes the third property stated in the Theorem. Properties (1) and (2) are proven in~\cite{Mis}.
\end{proof}

\end{document}